\documentclass[11pt, letterpaper, reqno]{amsart}
\usepackage{amsmath, amsthm, amscd, amsfonts, amssymb, graphicx,tikz, color,xcolor,mathtools}
\usepackage[all]{xy}
\usepackage{hyperref}
\usepackage[T1]{fontenc} 
\usepackage{mathrsfs}
\usepackage{ wasysym}
\usepackage[inline]{enumitem}
\usepackage{hyperref}
\usepackage{multicol}
\usepackage{rotating}
\usepackage{graphicx,fancyhdr}
\usepackage{multirow}

\makeatletter
\newcommand{\leqnomode}{\tagsleft@true\let\veqno\@@leqno}
\newcommand{\reqnomode}{\tagsleft@false\let\veqno\@@eqno}
\makeatother

\newtheorem{theorem}{Theorem}[section]
\newtheorem*{mtheorem}{Theorem \ref{thm:graded-simple}}
\newtheorem{lemma}[theorem]{Lemma}
\newtheorem{cor}[theorem]{Corollary}
\newtheorem{conjecture}[theorem]{Conjecture}

\theoremstyle{definition}
\newtheorem{remark}[theorem]{Remark}

\newcommand{\pf}{\begin{proof}}
\newcommand{\epf}{\end{proof}}

\newcommand{\ad}{\operatorname{ad}}

\newcommand{\Cosoc}{ \operatorname{Cosoc}}

\newcommand{\Der}{\operatorname{Der}}
\newcommand{\End}{\operatorname{End}}

\newcommand{\Ind}{\operatorname{Ind}}
\newcommand{\intt}{\operatorname{int}}

\newcommand{\rk}{\operatorname{rk}}

\newcommand{\Q}{{\mathbb Q}}

\newcommand{\Z}{{\mathbb Z}}

\newcommand{\Ac}{\mathcal A}
\newcommand{\Bc}{\mathcal B}

\newcommand{\Ic} {\mathcal I}

\newcommand{\Kc}{\mathcal K}
\newcommand{\Lc}{\mathcal L}
\newcommand{\Nc}{\mathcal N}

\newcommand{\Ssr}{\mathscr S}
\newcommand{\Vs}{\mathscr V}

\newcommand{\fsl}{\mathfrak{sl}}

\newcommand{\fg}{\mathfrak{g}}
\newcommand{\fh}{\mathfrak{h}}

\newcommand{\fm}{\mathfrak{m}}

\newcommand{\fq}{\mathfrak{q}}
\newcommand{\fs}{\mathfrak{s}}

\newcommand{\Ad}{\mathrm{Ad}}

\newcommand{\Ann}{\mathrm{Ann}}

\newcommand{\diff}{\mathrm{d}}

\newcommand{\Free}{\operatorname{Free}}

\newcommand{\Image}{\mathrm{Im}}

\newcommand{\Ker}{\operatorname{Ker}}

\newcommand{\Supp}{\operatorname{supp}}

\begin{document}
\title[Noetherian enveloping algebras]{Noetherian enveloping algebras of simple graded Lie algebras}
\author[Andruskiewitsch]{Nicol\'as Andruskiewitsch}
\address[N.~Andruskiewitsch]{Facultad de Matem\'atica, Astronom\'ia y F\'isica,
\newline Universidad Nacional de C\'ordoba. CIEM -- CONICET. \newline
Medina Allende s/n (5000) Ciudad Universitaria, C\'ordoba, Argentina}
\email{nicolas.andruskiewitsch@unc.edu.ar}

\author[Mathieu]{Olivier Mathieu}
\address[O.~Mathieu]{CNRS and Universit\'e Claude Bernard Lyon
\newline
Institut Camille Jordan
UMR 5028 du CNRS, 
 \newline
69622 Villeurbanne Cedex, France}
\email{mathieu@math.univ-lyon1.fr}

\address{
\newline 
SUSTech, Shenzhen International Center for Mathematics, Shenzhen, China}

\begin{abstract}
It is shown that if the universal enveloping algebra of a simple $\Z^n$-graded  Lie algebra is Noetherian, then the Lie algebra 
is finite-dimen\-sional.
\end{abstract}

\thanks{
N.~A. was  partially supported by the Secyt (UNC),
CONICET (PIP 11220200102916CO), FONCyT-ANPCyT (PICT-2019-03660). O.M. was partially supported by
UMR 5208 du CNRS. Both authors were partially supported
by the Shenzhen International Center of Mathematics, SUSTech.}

\maketitle
\setcounter{tocdepth}{2}

\section{Introduction}  

Let $K$ be an algebraically closed field of characteristic 0. 
If a Lie algebra is finite-dimensional, then its enveloping algebra  is Noetherian (i.e. Noetherian on the right or on the left, which is equivalent for the enveloping algebras).
Whether the converse is true  has been asked by many authors, 
among them  R. Amayo and I. Stewart, see \cite[Question 27]{Amayo-Stewart}, 
 K. A. Brown, see \cite[Question B]{Brown-survey}, J. Dixmier, and V. Latyshev.
 Besides its intrinsic interest, this is an unavoidable question  
 in the problem of the classification of Noetherian Hopf algebras.
 S. Sierra and C. Walton stated this question as a Conjecture.

\begin{conjecture}\label{conjecture:enveloping-Noetherian} \cite{Sierra-Walton}
The universal enveloping algebra of an infinite-dimen\-sional Lie algebra is not Noetherian.
\end{conjecture}

Intuitively, since `large' Lie algebras satisfy the Conjecture, e.g.
the enveloping algebra of a free Lie algebra in two generators is not Noetherian,
one expects that a counterexample to the Conjecture, if any, should be in some
sense 'small'. In this direction, a breakthrough result was obtained in 2013 by Sierra and Walton. Recall that the {\it Witt algebra} is
$W(1)\coloneqq \Der K[t]$.

\begin{theorem}\label{thm:sierra-walton} \cite[Theorem 0.5]{Sierra-Walton}
The enveloping algebra of $W(1)$ is not Noetherian.
\end{theorem}

This result allows us to conclude that 
the enveloping algebra of an infinite-dimensional 
simple $\Z$-graded  Lie algebra of finite growth
is not Noetherian, by going over the classification of such 
Lie algebras obtained in \cite{Mathieu-inv2}. 

However there are neither classification results for the simple $\Z$-graded  Lie algebras of arbitrary growth, nor  for the simple $\Z^n$-graded  Lie algebras for $n\geq 2$. Nevertheless, the
following result will be established.

\begin{theorem}\label{thm:graded-simple}
The universal enveloping algebra of an infinite-dimen\-sional simple $\Z^n$-graded  Lie algebra is not Noetherian.
\end{theorem}

According to our convention below, it is assumed in Theorem \ref{thm:graded-simple}
that {\it  all dimensions of the homogeneous 
components of the $\Z^n$-graded  Lie algebras are finite}.

The proof of Theorem \ref{thm:graded-simple} is divided
into four parts (three of them involve the case $n=1$). 
Some parts use concrete results, namely the Theorem of Sierra and Walton and the classification results of the second author \cite{Mathieu-inv2}.

\section{Conventions and Preliminaries}\label{prelim}

\subsection{\it Conventions about graded vector spaces}

\

\medbreak
\noindent
In the whole paper, we will adopt the following convention.
A vector space $M$ endowed with a decomposition
$M=\oplus_{{\mathbf m}\in \Z^n} M_{\bf m}$ will be called a 
{\it $\Z^n$-graded} vector space {\it only if} all homogeneous components $M_{\bf m}$  are finite-dimensional. 

\smallbreak
A Lie algebra $\mathcal{L}$ endowed with a $\Z^n$-grading is called
a {\it $\Z^n$-graded Lie algebra} if we have
\begin{align*}
[\mathcal{L}_{\bf n},\mathcal{L}_{\bf m}]
&\subset \mathcal{L}_{\bf n+m} & \text{for any } \bf n, \bf m \in \Z^n.
\end{align*}

 A $\Z^n$-graded Lie algebra $\mathcal{L}$ 
of dimension $\geq 2$ without nontrivial proper $\Z^n$-graded ideals
is called a {\it simple $\Z^n$-graded} Lie algebra.
For example,  $\fsl(2)\otimes K[t,t^{-1}]$ 
is a simple $\Z$-graded Lie algebra, but it is not simple as a Lie algebra. 
The definitions of a
 \emph{ $\Z^n$-graded $\mathcal{L}$-module} and  a
 \emph{ simple $\Z^n$-graded $\mathcal{L}$-module}
are similar.

\subsection{\it Criteria for Noetherianity of enveloping algebras}

\

\medbreak
\noindent A {\it section} of a Lie algebra $L$ is a Lie algebra
$\fs$ isomorphic to  $\fq/\fm$ for some Lie subalgebra  $\fq\subset L$ 
and some ideal $\fm$ of $\fq$.

The following standard observations are  useful, see
\cite[Lemma 1.7]{Sierra-Walton} and \cite[Proposition 2.1]{Buzaglo}.

\begin{lemma}\label{basic} 
	Let $L$ be a Lie algebra such that $U(L)$ is Noetherian. 
	
	\begin{enumerate}[leftmargin=*,label=\rm{(\alph*)}]
		\item\label{acc}
		$L$ satisfies the ascending chain condition on Lie subalgebras.
		
		\smallbreak
		\item\label{fp}
		$L$ is finitely presented and $H_k(L)$ is finite-dimensional for any $k\geq 0$.
		
		\smallbreak
		\item\label{section} If $\fs $ is a section of $L$,
		then $U(\fs)$ is  also Noetherian. 
		
		\smallbreak
		\item\label{abelian-section} If $\fs $ is an abelian  section of $L$,
		then $\dim\fs < \infty$. 
		
		\smallbreak
		\item\label{sup} 
		If $L$ is a Lie subalgebra of finite codimension of some Lie algebra $L'$, then $U(L')$ is also Noetherian. \qed
	\end{enumerate}
\end{lemma}

\subsection{\it Examples of   enveloping algebras that are not Noetherian}

\

\medbreak
\noindent Lemma \ref{basic} allows us to deduce that many 
Lie algebras satisfy Conjecture \ref{conjecture:enveloping-Noetherian} 
from  Lie algebras that are already known to fulfill it, for instance:

\begin{enumerate}[leftmargin=*,label=\rm{(\roman*)}]
	
	\medbreak
	\item The free Lie algebra $\Free(Z)$ on a vector space  $Z$ of dimension $\geq 2$. 
	Indeed $U(\Free(Z))  \simeq T(Z)$ is not Noetherian.

	\medbreak
	\item \cite [Theorem 0.5]{Sierra-Walton} 
	The positive Witt algebra $W_+$. 
	By Lemma \ref{basic}\ref{sup}, this result is equivalent to the remarkable Theorem
	\ref{thm:sierra-walton}.
\end{enumerate}

\medbreak
See \cite{Sierra-Walton}, \cite{Buzaglo} for a list of Lie algebras  
whose enveloping algebras are not Noetherian by the remarks above. 
By Lemma \ref{basic} \ref{section} and (i), 
another example is  a Kac–Moody algebra of indefinite type, 
cf. \cite[Corollary 9.12]{Kac-libro}.

\section{Growth of modules over  \texorpdfstring{$\Z$}{}-graded Lie algebras}\label{Zgrading}

\noindent In this section and in the next three, we investigate the Noetherianity condition for $\Z$-graded Lie algebras.
The present section involves the questions of finite generation and growth.

\smallbreak
Given  a $\Z$-graded vector space $M$ and an integer $n\in\Z$, we set
\begin{align*}
M_{\geq n}\coloneqq \oplus_{k\geq n}\,M_k.
\end{align*}
 The subspaces $M_{>n}$, $M_{\leq n}$ and 
$M_{< n}$ are similarly defined.

\subsection{\it Finite generation}

\

\medbreak
\noindent Let $\mathcal{L}$ be a $\Z$-graded Lie algebra.
We set
\begin{align*}\mathcal{L}^+=\mathcal{L}_{>0} \quad \text{ and } \quad
\mathcal{L}^-=\mathcal{L}_{<0}.
\end{align*}

\begin{lemma}\label{fg} 
Assume that the Lie algebra $\mathcal{L}$ is finitely generated.
Then $\mathcal{L}^+$ and $\mathcal{L}^-$ are finitely generated subalgebras.

Moreover let $M$ be a finitely generated $\Z$-graded 
$\mathcal{L}$-module. Then the $\mathcal{L}^+$-module $M_{\geq 0}$  and the  $\mathcal{L}^-$-module $M_{\leq 0}$ are finitely generated.
\end{lemma}

\begin{proof} By hypothesis, there is an integer $d>0$ such that $\oplus_{-d\leq k\leq d}\,\mathcal{L}_k$ generates
$\mathcal{L}$. By Lemma 18 of \cite{Mathieu-inv1},
$\oplus_{1\leq k\leq d}\,\mathcal{L}_k$ generates
$\mathcal{L}^+$ and $\oplus_{-d\leq k\leq -1}\,\mathcal{L}_k$
generates $\mathcal{L}^-$, which proves the first assertion.

Let $S$ be a finite set of generators of $M$. There is an 
integer $e$ such that $S$ lies in $M_{\leq e}$. Since
$M_{\leq e}$ is a $\mathcal{L}_{\leq 0}$-module, we have
$M=U(\mathcal{L}^+). M_{\leq e}$. Since in addition
$\mathcal{L}^+$ is generated by
$\oplus_{1\leq k\leq d}\,\mathcal{L}_k$, we have

\begin{align*}M_n=\sum_{1\leq k\leq d}\, 
\mathcal{L}_k. M_{n-k},
\end{align*}
\noindent for any $n>e$. It follows easily that 
the $\mathcal{L}^+$-module
 $M_{\geq 0}$ is finitely generated.
The proof of the finite generation of  the  $\mathcal{L}^-$-module $M_{\leq 0}$  is identical. 
\end{proof}

\subsection{\it  Finite and intermediate growth}

\

\medbreak
\noindent A $\Z$-graded vector space $M$ is
called of {\it finite growth} if the function
\begin{align*}n\mapsto \dim\,M_n
\end{align*}
\noindent is bounded by a polynomial.
It is called of {\it intermediate growth} if both limits
\begin{align*}
&\limsup \frac{\log^+(\dim M_n)}{ n}&  &\text{and}
&&\limsup \frac{\log^+(\dim M_{-n})} {n}
\end{align*}
 are  zero, where the function $\log^+$ is 
defined by  $\log^+(x)=\log(x)$ if $x\geq 1$ and $\log^+(x)=0$
otherwise. The  formal series

\begin{align*}
\chi_M^{\pm}(z)\coloneqq \sum_{n\geq 0}\,\dim\,M_{\pm n}\,z^n 
\end{align*}
 are called the {\it two generating series} of $M$. Equivalently, $M$ has intermediate growth iff  both series 
$\chi_M^{+}(z)$ and $\chi_M^{-}(z)$ 
are convergent for $\vert z\vert<1$.

\medbreak
Assume now that $M=\oplus_{n\geq 1}\, M_n$ is a positively graded
vector space. Then the symmetric algebra $S(M)$ is a nonnegatively graded vector space.
The following lemma is well-known.

\begin{lemma}\label{interS} Assume that the positively graded vector space
$M$ has intermediate growth. Then $S(M)$ also has  intermediate growth.
\end{lemma}

\begin{proof} For any integer $n\geq 1$, set $a_n=\dim \, M_n$.
We have

\begin{align*}
\chi^+_{M}(z) &= \sum_{n\geq 1} \, a_n z^n, &
\chi^+_{S(M)}(z) &= \prod_{n\geq 1} \, \frac{1}{(1-z^n)^{a_n}}.
\end{align*}

\noindent If $M$ is finite-dimensional, $S(M)$ has finite growth. Otherwise, the lemma follows because 
these  series have the same radius of convergence.
\end{proof}

\subsection{\it Growth of \texorpdfstring{$\Z$}{}-graded \texorpdfstring{$\mathcal{L}$}{}-modules}

\

\medbreak
\noindent
Let $\mathcal{L}$ be a $\Z$-graded Lie algebra and
let $M$ be a $\Z$-graded $\mathcal{L}$-module.
For any $n$, let $M_n^{\intt}$ be the subspace of 
all $m\in M_n$ such that $U(\mathcal{L}^+). m$ 
has intermediate growth. Set
$M^{\intt}=\oplus_{n\in\Z} M_n^{\intt}$.

\begin{lemma}\label{submodule} The subspace $M^{\intt}$ is a $\mathcal{L}$-submodule.
\end{lemma}

\begin{proof} Since $M^{\intt}$  is clearly  a 
$\mathcal{L}^+$-module, it is enough to show that
for any homogeneous elements $u\in \mathcal{L}$ of degree $d\leq 0$
and $v\in M^{\intt}$,   $u . v$ belongs to  $M^{\intt}$.
First note that
\begin{align*}
U(\mathcal{L}^+)u\subset U(\mathcal{L}^+)\mathcal{L}_{\geq d}
=\mathcal{L}_{\geq d}U(\mathcal{L}^+)
=U(\mathcal{L}^+)\oplus\oplus_{d\leq k\leq 0}\, \mathcal{L}_{k}U(\mathcal{L}^+). 
\end{align*}
\noindent Therefore we have 
\begin{align*}
U(\mathcal{L}^+)u.v \subset U(\mathcal{L}^+).v+
\sum_{d\leq k\leq 0}\, \mathcal{L}_{k}U(\mathcal{L}^+).v.
\end{align*}
Thus $U(\mathcal{L}^+)u.v$ has intermediate growth, i.e. $u . v$ belongs to $M^{\intt}$. 
\end{proof}

\begin{lemma}\label{interM} Let $\mathcal{L}$ be a finitely generated 
$\Z$-graded Lie algebra and let $M$ be a simple $\Z$-graded $\mathcal{L}$-module.
Assume that, for some homogeneous $v\in M\setminus 0$, the vector space $\mathcal{L}. v$ has intermediate growth.
Then $M$ has intermediate growth.
\end{lemma}

\begin{proof} Let $\Kc^+=\{x\in \mathcal{L}^+\mid x.v=0\}$.
As a graded space, the ${\mathcal{L}^+}$-module 
$\Ind_{\Kc^+}^{\mathcal{L}^+} Kv$ is isomorphic to 
$S({\mathcal{L}^+}/\Kc^+)$. By Lemma \ref{interS},
$\Ind_{\Kc^+}^{\mathcal{L}^+} Kv$ has in\-ter\-me\-diate growth. 
Thus $U(\mathcal{L}^+). v$, a quotient of 
$\Ind_{\Kc^+}^{\mathcal{L}^+} Kv$, 
 has intermediate growth too.

Since $M$ is simple, from  Lemma \ref{submodule} we infer  that any cyclic 
$U(\mathcal{L}^+)$-sub\-module of $M$ has intermediate growth. 
Now the $\mathcal{L}^+$-module $M_{\geq 0}$ is finitely generated 
by Lemma \ref{fg},  hence $M_{\geq 0}$ has intermediate growth. 
 Similarly $M_{\leq 0}$ has intermediate growth; therefore $M$
 has intermediate growth.
\end{proof}

\section{Rank one Lie algebras of class \texorpdfstring{$\Vs$}{}}\label{V}

\noindent We define first the general notions of roots and rank
of a $\Z$-graded Lie algebra $\Lc$.  Then we split the proof that $U(\mathcal{L})$  is not Noetherian into three parts:
Lie algebras  of class $\Vs$ are treated in this section;
the next section \ref{S} is devoted to  class 
$\Ssr$; the last section \ref{rk2} deals with the
Lie algebras  of rank $\geq 2$.

\subsection{\it  Roots and rank}

\

\medbreak
\noindent
Let $\mathcal{L}=\oplus_{n\in\Z} \Lc_{n}$ be a $\Z$-graded Lie algebra.  
We fix, once and for all,  a Cartan subalgebra $\fh$ of $\mathcal{L}_0$, 
i.e.,  $\fh$ is a nilpotent self-normalizing subalgebra of $\mathcal{L}_0$
\cite{Bourbaki}. For any $\alpha\in\fh^*$ and any $n\in\Z$, we set
\begin{align*}
\mathcal{L}_n^\alpha=
\{x\in \mathcal{L}_n\mid \, (\ad(h)-\alpha(h))^N(x) = 0 \quad
\forall h\in\fh \text{ and } N\gg 0\}.
\end{align*}
 Also, $\mathcal{L}^{\widetilde\alpha} \coloneqq \mathcal{L}^{\alpha}_n$
 for $\widetilde\alpha=(\alpha,n)\in \fh^*\times\Z$.
The {\it set of roots}  of 
$\mathcal{L}$ is the set
$\Delta\coloneqq \{\widetilde{\alpha}\mid \mathcal{L}^{\widetilde\alpha}\neq 0\}$
(with our nonstandard definition, $(0,0)$ is a root
whenever ${\mathcal L}_0\neq 0$).
Therefore 
\begin{align*}
\mathcal{L}
=\oplus_{\widetilde{\alpha}\in\Delta}\,\mathcal{L}^{\widetilde\alpha}
\end{align*}
 is the generalized root space decomposition
of $\mathcal{L}$. 

\smallbreak
A root $\widetilde{\alpha}=(\alpha,n)$ is called
{\it real} if $\alpha\neq 0$ and  {\it imaginary} otherwise. Let $\Delta_{\operatorname{re}}$, respectively
$\Delta_{\operatorname{im}}$, be the set of real, respectively imaginary, roots.

\smallbreak
The {\it root lattice} is the subgroup
$Q\subset \fh^*\times \Z$ generated by $\Delta$. 
By definition the {\it rank} of ${\mathcal L}$ is the rank of $Q$. 

\smallskip
\begin{remark}
It is proved in \cite{Bourbaki} that any two Cartan subalgebras of ${\mathcal L}_0$ are conjugated by an automorphism of $\mathcal{L}_0$. In fact the proof of
\cite{Bourbaki} shows that they  are conjugated by a degree-preserving automorphism of ${\mathcal L}$. Therefore the root lattice is independent of the choice of a Cartan subalgebra of ${\mathcal L}_0$. Since we do not need
this fact, we will not provide more details.
\end{remark}

\subsection{\it  Rank one Lie algebras}

\

\medbreak
\noindent
Let
$\mathcal{L}=\oplus_{n\in\Z} \Lc_{n}$  be a 
$\Z$-graded Lie algebra of rank one. Assume that
${\mathcal L}\neq {\mathcal L}_0$. Then 
${\mathcal L}_0=\fh$ is a nilpotent Lie algebra and there exists
$\widetilde\alpha=(\alpha,1)\in{\mathcal L}_0^*\times\Z$ such that 
$\Delta$ lies in $\Z.\widetilde\alpha$. 
We  keep the terminology of \cite{Mathieu-inv2}.
When $\alpha=0$ or, equivalently, when the set
of real roots is void, 
we say that $\mathcal{L}$ belongs to the class 
$\Vs$ (for the class $\Vs$,  the Lie algebra $\Lc_0$ could be $0$). Otherwise, we say that 
$\mathcal{L}$ belongs to the class $\Ssr$.
Here the letter $\Ssr$ stands for string, because,
roughly speaking, all real roots are on a ``string''.

\subsection{\it  Rank one Lie algebras of class \texorpdfstring{$\Vs$}{}}

\

\medbreak
\noindent
This case follows easily from the next result.

\begin{lemma}\label{nfg} \cite[Lemma 22]{Mathieu-jalg}
Let $\mathcal{L}$ be a $\Z$-graded Lie algebra of class 
$\Vs$. If $\mathcal{L}=[\mathcal{L},\mathcal{L}]$,  then 
$\mathcal{L}$ is not finitely generated.
\qed
\end{lemma}

\begin{cor}\label{corV} Let $\mathcal{L}$ be a simple 
$\Z$-graded Lie algebra of
class $\Vs$. Then $U(\mathcal{L})$ is not Noetherian.
\end{cor}

\begin{proof} Immediate from Lemmas  \ref{basic}\ref{fp} and \ref{nfg}.
\end{proof}

\section{Rank one Lie algebras of class \texorpdfstring{$\Ssr$}{}}\label{S}

\noindent The case of Lie algebras of class 
$\Ssr$ is more difficult than the previous one.
Recall that a $\Z$-graded Lie algebra $\mathcal{L}$ belongs to $\Ssr$ if ${\mathcal L}\neq {\mathcal L}_0$, 
${\mathcal L}_0$ is nilpotent and
there exists 
a nonzero $\alpha\in{\mathcal L}_0^*$ such that 
$\mathcal{L}_n=\mathcal{L}_n^{n\alpha}$ for any $n\in\Z$.

\medbreak
The main step is Theorem \ref{thmS}, which is implicit in \cite{Mathieu-inv2}. 
Navigating through chapters 7 and 8 of {\it loc. cit.} is not easy. Thus
for the sake of the reader, we rewrite parts of those  in a  convenient way.

\smallbreak
We need the following definition.
For  $n\neq 0$, let $\mathcal{L}\{n\}$ be the Lie algebra
$\mathcal{L}$ endowed with a grading rescaled by
a factor of $n$, i.e. we have
\begin{align*}
\mathcal{L}\{n\}_{nk} &= \Lc_{k}, \quad k \in\Z,&
\mathcal{L}\{n\}_{m} &= 0 \text{ if } n \not\vert m. 
\end{align*}
The $\Z$-graded Lie algebra 
$\mathcal{L}\{n\}$,  again in class 
$\Ssr$,  is called a {\it rescaling of 
$\mathcal{L}$}.

\subsection{\it  Local Lie algebras}

\

\medbreak
\noindent Let $P$ be the set of pairs of integers $(i,j)$ with
$i,j, i+j\in\{-1,0,1\}$. Following \cite{Kac}, see  \cite[Exercise 1.8, p. 13]{Kac-libro},
a {\it local Lie algebra} is a graded vector space 
\begin{align*}
G=G_{-1}\oplus G_0\oplus G_1
\end{align*}
endowed with a degree preserving  bracket $[\,, \,]$ which is  defined only
on $\cup_{(i,j)\in P}\, G_i\times G_j$ and which satisfies
the Jacobi identity whenever it makes sense. Equivalently,
this means that $G_0$ is a Lie algebra,
$G_1$ and $G_{-1}$ are $G_0$-modules and the bracket
$[,]:G_{-1}\times G_1\to G_0$ is $G_0$-equivariant.

\medbreak
The notions of morphisms between local Lie algebras, local
Lie subalgebras and local ideals are defined in an evident way.
Analogously a local Lie algebra $S$ is a called a {\it section}  of $G$
if $S$ is isomorphic to $H/K$ for some local subalgebra $H\subset G$ 
and some local ideal $K$ of $H$.

\medbreak
Given a $\Z$-graded Lie algebra $\mathcal{L}$, its 
{\it local part} 
\begin{align*}
\mathcal{L}_{loc}:
=\mathcal{L}_{-1}\oplus\mathcal{L}_0\oplus\mathcal{L}_1
\end{align*}
is evidently a local Lie algebra. Conversely, given a local
Lie algebra $G$ there  are $\Z$-graded Lie algebras
whose local part is $G$. One of them,
denoted by $\mathcal{L}_{\max}(G)$, is defined as follows.
As a vector space we have 
\begin{align*}
\mathcal{L}_{\max}(G)=\Free(G_{-1})\oplus G_0\oplus \Free(G_1)
\end{align*}
 where $\Free(G_{\pm1})$ denotes the free Lie algebra on the vector space $G_{\pm1}$. 
 Then the  local Lie bracket and the $\Z$-grading extend uniquely  to 
$\mathcal{L}_{\max}(G)$ \cite{Kac}. 
Indeed the functor $G\to \mathcal{L}_{\max}(G)$ is the left adjoint of the functor $\mathcal{L}\to\mathcal{L}_{loc}$ \cite{Mathieu-inv1}. Let 
$\mathcal{I}$ be the largest
graded ideal of $\mathcal{L}_{\max}(G)$ such that
$\mathcal{I} \cap G = 0$ and set
\begin{align*}\mathcal{L}_{\min}(G)=\mathcal{L}_{\max}(G)/\mathcal{I}.
\end{align*}
Notice that, if $\mathcal{L}$ is a
Lie $\Z$-graded Lie algebra  which is generated by its local part $G$,  
then there are natural epimorphisms
\begin{align*}\mathcal{L}_{\max}(G)& \twoheadrightarrow\mathcal{L}
&&\text{ and }&\mathcal{L} &\twoheadrightarrow \mathcal{L}_{\min}(G),
\end{align*}

\noindent so $\mathcal{L}$ is between the Lie algebras
$\mathcal{L}_{\max}(G)$ and $\mathcal{L}_{\min}(G)$. We conclude:

\begin{lemma}\label{local} Let $G$ be a local Lie algebra and let
$\mathcal{L}$ be a $\Z$-graded Lie algebra.
If $G$ is a section of $\mathcal{L}_{loc}$, then
$\mathcal{L}_{\min}(G)$ is a section of $\mathcal{L}$. \qed
\end{lemma}

\subsection{\it  Four basic simple Lie algebras of class \texorpdfstring{$\Ssr$}{}}

\

\medbreak
\noindent We  start recalling the definitions of some Lie algebras of class $\Ssr$.

\medbreak
\begin{itemize}[leftmargin=*]  

\item The {\it centerless Virasoro algebra} is  $W=\Der K[t,t^{-1}]$.
It has a natural grading, relative to which the element $L_n:=t^{n+1}\frac{\diff} {\diff t}$ is homogeneous of degree $n$. We
have $\fh=K. L_0$.

\medbreak
\item The {\it Witt algebra}  is $W(1)=\Der K[t]$; it is a graded subalgebra of $W$.

\medbreak
\item The Lie algebra $\fsl(2)$; it is the Lie subalgebra
of $W$ with basis $\{L_{-1}, L_0, L_1\}$. 

\medbreak
\item The contragredient Lie algebra $G(^{2\,2}_{2\,2})$. It is generated by five elements
$h, e_1, e_2, f_1, f_2$ and defined by the following relations
\begin{align} \label{eq:local-G}
[h,e_i] &= 2 e_i, & [h,f_i] &= -2 f_i, & [e_i,f_j] &= \delta_{i,j}\,h,
\end{align}
\noindent for any $i,j\in\{1,2\}$, where, as usual, $\delta_{i,j}$ is the Kronecker symbol. 
It has a $\Z$-grading relative to which 
the $e_i$'s  have degree one, $h$ has degree zero and
the $f_i$'s have degree $-1$.
\end{itemize}
\medbreak

These four $\Z$-graded Lie algebras and their rescalings 
are simple Lie algebras of class $\Ssr$ (for
the simplicity of $G(^{2\,2}_{2\,2})$ see
\cite{Kac}). They play a central role in what follows.

\subsection{\it  Non-Abelian free subalgebras of \texorpdfstring{$\Z$}{}-graded Lie algebras}

\

\medbreak

\noindent 
Let 
$G(^{2\,2}_{2\,2})_{loc}$ be the local part of the Lie algebra 
$G(^{2\,2}_{2\,2})$. 
Since $G(^{2\,2}_{2\,2})$ is generated by its local part and
is defined by local relations, we have
\begin{align*}G(^{2\,2}_{2\,2})=
\mathcal{L}_{\max}(G(^{2\,2}_{2\,2})_{loc}).\end{align*}

\begin{lemma}\label{Kac}
We have $G(^{2\,2}_{2\,2})=
\mathcal{L}_{\min}(G(^{2\,2}_{2\,2})_{loc})$.
\end{lemma}

\begin{proof} This follows because the Lie algebra
$G(^{2\,2}_{2\,2})$ is simple \cite{Kac}.
\end{proof}

\begin{lemma}\label{free}
Let $\Lc$ be a $\Z$-graded Lie algebra. If
$G(^{2\,2}_{2\,2})_{loc}$ is a section of
$\Lc_{loc}$, the Lie algebra $\Lc$ contains a 
non-Abelian free Lie subalgebra.
\end{lemma}

\begin{proof} 
By Lemmas \ref{local} and \ref{Kac},
$G(^{2\,2}_{2\,2})$ is a section of $\Lc$.
The Lie subalgebra of $G(^{2\,2}_{2\,2})$ generated
by $e_1$ and $e_2$ is free of rank two. Hence 
$\Lc$ admits a 
non-Abelian free section, which can be lifted to
a Lie subalgebra of $\Lc$.
\end{proof}

\subsection{\it  A simple criterion for a section isomorphic to \texorpdfstring{$G(^{2\,2}_{2\,2})$}{}}
\
\medbreak
\noindent Let $\mathcal{L}$ be a $\Z$-graded Lie algebra of class $\Ssr$
with $\alpha\in{\mathcal L}_0^*$ as above. 
In this subsection and in the next two, we do not assume that $\mathcal{L}$ is simple as a $\Z$-graded algebra.
We will describe  criteria for the existence of a section
of $\Lc$  isomorphic to $G(^{2\,2}_{2\,2})$.

\medbreak
 For $n\neq 0$, let $B_n: \mathcal{L}_{-n}\times\mathcal{L}_n \to K$ be the bilinear map
\begin{align*}B_n:(x,y)\in\mathcal{L}_{-n}\times\mathcal{L}_n
\mapsto\alpha([x,y]).
\end{align*}
 Let $\Kc_n$ and 
$\Kc_{-n}$ be its right kernel and its left kernel.
 Since $\alpha([{\mathcal L}_0,{\mathcal L}_0])=0$,
 the bilinear map $B_n$ is ${\mathcal L}_0$-equivariant.
 Also $\Kc_0:=\Ker \alpha$ is an ideal of
 ${\mathcal L}_0$. 
 
 Given a finite-dimensional ${\mathcal L}_0$-module
 $M$, its {\it cosocle} is its maximal semi-simple quotient. It is denoted by $\Cosoc\,M$. 
 For any $\beta\in (\Lc_0/[\Lc_0,\Lc_0])^*$, let $K_\beta$
be the one-dimensional $\Lc_0$-module on which 
any  $h\in\Lc_0$ acts by
multiplication by $\beta(h)$. Obviously, $\Cosoc\,\Lc_n/ \Kc_n$ is direct sum of  copies of $K_{n\alpha}$.

\begin{lemma}\label{2gen} Assume that 
$\Cosoc\,\Lc_n/ \Kc_n$ has dimension $\geq 2$ for some $n>0$.
Then, up to a rescaling, $G(^{2\,2}_{2\,2})$ is a section
of $\mathcal{L}$.
\end{lemma}

\begin{proof} 
Since we did not assume that $\mathcal{L}$ is simple, 
we can assume that $n=1$. By assumption, there is
a $\mathcal{L}_0$-module $\mathcal{I}_1$ with
$\Kc_1\subset\mathcal{I}_1\subset \mathcal{L}_1$ such that
$\mathcal{L}_1/\mathcal{I}_1$ is isomorphic to
$K_\alpha\oplus K_\alpha$.
Let $\mathcal{L}'_{-1}\subset \mathcal{L}_{-1}$ be the orthogonal of $\mathcal{I}_1$  with respect to the bilinear
map $B_1$.

By definition $\mathcal{L}'_{-1}$ contains $\Kc_{-1}$
and the $\Lc_0$-module $\mathcal{L}'_{-1}/\Kc_{-1}$
is isomorphic to 
$K_{-\alpha}\oplus K_{-\alpha}$.
It follows that 
 $\mathcal{I}\coloneqq \Kc_{-1}\oplus \Kc_0\oplus {\mathcal I}_1$ is a local 
ideal of the local Lie algebra 
$\mathcal{G}\coloneqq \mathcal{L}'_{-1}\oplus \mathcal{L}_0
\oplus \mathcal{L}_1$. Since
$\mathcal{G}/\mathcal{I}$ is clearly isomorphic to the local part of
$G(^{2\,2}_{2\,2})$, 
it follows from Lemmas \ref{local} and \ref{Kac} that
$G(^{2\,2}_{2\,2})$ is a section of $\mathcal{L}$. 
\end{proof}

\subsection{\it Cyclic modules}

\

\medbreak
\noindent
Let $\fg$ be a   Lie algebra 
and let $M$ be a finite-dimensional 
$\fg$-module.

\begin{lemma}\label{cyclic} There is a polynomial
$P$ such that  the dimension of any
 cyclic $U(\fg)$-module in $M^{\otimes n}$ is bounded by
$P(n)$, for any $n\geq 1$.
\end{lemma}

\begin{proof} For any positive integer $n$, let
$I_n$ be the annihilator of the $U(\fg)$-module 
$M^{\otimes n}$. Since the coproduct is cocommutative,
$U(\fg)/I_n$ embeds into $H^0(S_n, \End(M)^{\otimes n})$,
where the symmetric group $S_n$ acts on 
$\End(M)^{\otimes n}$ by permutation of the factors.
Thus for any cyclic $U(\fg)$-submodule $C$ of 
$M^{\otimes n}$,  the inequality
\begin{align*}
\dim C\leq \dim \, U(\fg)/I_n \leq 
\dim S^n (\End(M))= {\textstyle\binom{n+(\dim M)^2-1}{n}} 
\end{align*}
\noindent is a polynomial bound  of degree
$(\dim V)^2-1$. 
\end{proof} 

\subsection{\it  An improved criterion for a section isomorphic to \texorpdfstring{$G(^{2\,2}_{2\,2})$}{}}

\
\medbreak
\noindent Using the notation  of the previous 
sections, we show that the conclusion of Lemma \ref{2gen} holds with a weaker hypothesis.

\begin{lemma}\label{G(A)} Assume that $\mathcal{L}_n/\Kc_n$ has dimension
$\geq 2$ for some
$n>0$. Then, up to a rescaling, $G(^{2\,2}_{2\,2})$ is a section of $\mathcal{L}$.
\end{lemma}

\begin{proof} We can assume that $n=1$ and that 
$\mathcal{L}$ is generated by its local part.
This implies that $\mathcal{L}^+$ is generated by
$\mathcal{L}_1$.

By Lemma \ref{cyclic},  
the dimensions of the cyclic $\Lc_0$-modules in
$\mathcal{L}_1^{\otimes n}$ are boun\-ded by a polynomial 
on $n$. Obviously, 
 the same  property holds for its quotient
$\mathcal{L}_n/\Kc_n$. 
However, by Lemma 7.9 of \cite{Mathieu-inv2}, the function
$n\mapsto\rk\,B_n=\dim \mathcal{L}_n/\Kc_n$ has infinite growth (i.e. it is not bounded by a polynomial). Hence,
when $n$ goes to $\infty$, the minimal
number of generators of the $\Lc_0$-module $\mathcal{L}_n/\Kc_n$
is arbitrarily large.
Thus the function $n\mapsto 
\dim\,\Cosoc(\mathcal{L}_n/\Kc_n)$ is unbounded.

Therefore, for some $n$, we have  
\begin{align*}
\dim\,\Cosoc(\mathcal{L}_n/\Kc_n) &\geq 2.
\end{align*}
Thus by Lemma \ref{2gen}, 
$G(^{2\,2}_{2\,2})$ is a section of $\mathcal{L}$.
\end{proof}

\subsection{\it  The dichotomy for the class \texorpdfstring{$\Ssr$}{}}

\

\medbreak  
\noindent From now on, we assume that the Lie algebra $\mathcal{L}$ of class $\Ssr$ is simple.
 It is implicitly proved in \cite[Chapter 8]{Mathieu-inv2} 
 that $\mathcal{L}$ is isomorphic to $W$ or $W(1)$, under the hypothesis
 \leqnomode
\begin{align}\tag{$\mathcal{H}_1$} 
\text{all bilinear forms $B_n$ have rank $\leq 1$.}
\end{align}
 Unfortunately, the explicit hypothesis used in \cite[Chapter 8]{Mathieu-inv2} is 
\begin{align}\tag{$\mathcal{H}_2$} 
\text{the Lie algebra $\mathcal{L}$ has intermediate growth.}
\end{align}

It would be long to go into the details of \emph{loc. cit.}
to explain why $(\mathcal{H}_1)$ can be used instead of $(\mathcal{H}_2)$. 
Here we can assume that 
$\mathcal{L}$ is finitely generated. 
Under this additional hypothesis,  the next lemma gives an easy explanation.

\begin{lemma}\label{H1H2} If $\mathcal{L}$ is finitely generated,
then  $(\mathcal{H}_1)$ implies $(\mathcal{H}_2)$.
\end{lemma}

\begin{proof} Let 
$M\coloneqq \oplus_{n\in\Z}\,\mathcal{L}_n^*$ be the graded dual of the adjoint module. The hypothesis $(\mathcal{H}_1)$
means that  the $\Z$-graded space $\mathcal{L}.\alpha$ has
homogenous components of dimension $\leq 1$. By Lemma \ref{interM}, we see that $M$ has intermediate growth, i.e.  $(\mathcal{H}_2)$ holds.
\end{proof}

The following result is implicitly proved in \cite{Mathieu-inv2}, even without the hypothesis of finite generation.

\begin{theorem}\label{thmS} Let $\mathcal{L}$ be a simple
$\Z$-graded Lie algebra of class $\Ssr$.
Assume that $\mathcal{L}$ is finitely generated. Then
\begin{enumerate} [leftmargin=*,label=\rm{(\roman*)}]
\item either $\mathcal{L}$ is isomorphic to $\fsl(2)$, $W(1)$ or $W$,

\item or $\mathcal{L}$  contains a nonabelian free Lie algebra.
\end{enumerate}
\end{theorem}

\begin{proof} (i) First assume that the bilinear form
$B_n$ has rank $\leq 1$ for any $n$. By Lemma \ref{H1H2},
 $\mathcal{L}$ has intermediate growth. Thus it follows from Proposition 8.9
of \cite{Mathieu-inv2} that $\mathcal{L}$ is isomorphic to
$\fsl(2)$, $W(1)$ or $W$.

\smallbreak
(ii) Otherwise, the bilinear form has rank 
$\geq 2$ for some $n$. By Lemma \ref{G(A)} the Lie algebra 
$G(^{2\,2}_{2\,2})$ is a section of $\mathcal{L}$. 
By Lemma \ref{free}, the Lie
algebra 
$\mathcal{L}$  contains a nonabelian free Lie algebra.
\end{proof}

\begin{cor}\label{corS} Let $\mathcal{L}$ be a simple
$\Z$-graded Lie algebra of class $\Ssr$. 
Then $U(\mathcal{L})$ is not Noetherian, except if
$\mathcal{L}$ is isomorphic to $\fsl(2)$.
\end{cor}

\begin{proof} Assume that $\mathcal{L}$ is infinite-dimensional. 
By Lemma \ref{basic}\ref{fp}, we can assume that $\Lc$ is finitely generated.
By Theorem \ref{thmS}, $\mathcal{L}$
contains a subalgebra isomorphic to the Witt algebra
$W(1)$ or  a nonabelian free subalgebra.
In the first case, Theorem \ref{thm:sierra-walton} implies
that $U(\mathcal{L})$ is not Noetherian. In the second case,
we already observed that the enveloping algebra of a nonabelian free Lie algebra
is not Noetherian, so neither is $U(\mathcal{L})$.
\end{proof}

\section{\texorpdfstring{$\Z$}{}-graded Lie algebras of rank \texorpdfstring{$\geq2$}{}}\label{rk2}

\noindent In this section  we investigate the Noetherianity condition for $\Z$-graded Lie algebras of rank $\geq 2$.

\subsection{\it  Weakly \texorpdfstring{$\Z$}{}-graded Lie algebras}

\

\medbreak

\noindent
We will encounter Lie algebras 
$\mathcal{M}$  with a decomposition
$\mathcal{M}=\oplus \mathcal{M}_n$ satisfying
$[\mathcal{M}_n,\mathcal{M}_m]\subset
\mathcal{M}_{n+m}$ where the homogeneous components
could be of infinite dimension; we shall
 call them {\it weakly $\Z$-graded} Lie algebras.

 \medbreak
\begin{lemma}\label{wgraded}
Let $\mathcal{L}=\oplus_{n\in\Z}\,\mathcal{L}_n$ be a  weakly $\Z$-graded  Lie algebra such that $U(\mathcal{L})$ is Noetherian.
Then  $\mathcal{L}_n$ is finite-dimensional
for any $n \in \Z\setminus 0$. 

Moreover if $\mathcal{L}$ is simple and $\mathcal{L} \neq\mathcal{L}_0$, then  $\mathcal{L}_0$ is also  finite-dimensional. 
\end{lemma}

\begin{proof} As before, set
$\Lc^+=\mathcal{L}_{>0}$ and 
$\Lc^-=\mathcal{L}_{<0}$.

By Lemma \ref{basic}, $U(\Lc^+)$ is Noetherian,
hence  $\Lc^+$  is finitely generated. It follows easily
 that all homogeneous components of 
$\Lc^+$ are finite-dimensional. 
Similarly, all homogenous components of $\Lc^-$ are finite-dimensional, which proves the first assertion.

Assume now that $\mathcal{L}$ is simple and $\mathcal{L} \neq\mathcal{L}_0$.
Since $\Lc^+$ and $\Lc^-$ are finitely generated, there is an integer $d>0$ such that
$\oplus_{1\leq i\leq d}\,\Lc_i$ generates $\Lc^+$ and
$\oplus_{1\leq i\leq d}\,\Lc_{-i}$ generates $\Lc^-$. Set
\begin{align*}
M & = \oplus_{1\leq i\leq d}\,\Lc_i \, \bigoplus \, 
\oplus_{1\leq i\leq d}\,\Lc_{-i}
\end{align*}
and let $\Kc$ be the annhilator in 
$\Lc_0$ of the $\Lc_0$-module $M$. Since $M$ is a finite-dimensional $\Lc_0$-module,
$\Lc_0/\Kc=0$ is finite-dimensional.
We have 
\begin{align}
[\Kc,\Lc^{\pm}]=0, 
[\Kc,\Lc_0]\subset \Lc_0 \text{, and } \Kc\subset\Lc_0.
\end{align} 

\noindent Hence $\Kc$ is a proper ideal of $\Lc$. Since $\Lc$ is simple, the ideal $\Kc$ is trivial. It follows that 
$\Lc_0$ is finite-dimensional. 
\end{proof}

\subsection{\it  The hypothesis \texorpdfstring{$(\mathcal{H})$}{}}

\

\medbreak
\noindent Let $\mathcal{L}$ be a $\Z$-graded Lie algebra.
Consider the following hypothesis 
\begin{align}\tag{$\mathcal{H}$} 
\quad \text{There exist } &\widetilde{\alpha},\widetilde{\beta}\in Q, \,
\widetilde{\beta} \notin \Q.\widetilde{\alpha}, \, \text{such that} \,
({\widetilde\beta}+\Z.{\widetilde\alpha}) \cap \Delta \, \text{is infinite.}
\end{align}

\begin{lemma}\label{H1} If  $\mathcal{L}$ satisfies the hypothesis $(\mathcal{H})$,
then $U(\mathcal{L})$ is not Noetherian.
\end{lemma}

\begin{proof} For any integer $k\geq 1$, set 
$\Delta(k)=(k.{\widetilde\beta}+\Z.{\widetilde\alpha})\cap \Delta$ and 
\begin{align*}\mathcal{M}_k
=\oplus_{\widetilde\gamma\in \Delta(k)}\,\mathcal{L}^{\widetilde\gamma}.
\end{align*}
 Since we have $[\mathcal{M}_k,\mathcal{M}_l]\subset
\mathcal{M}_{k+l}$ for any $k,l\geq 1$,
the vector space 
\begin{align*}\mathcal{M}\coloneqq \oplus_{k\geq 1}\,\mathcal{M}_k
\end{align*} 
is a weakly $\Z$-graded Lie algebra. Since 
$\widetilde{\alpha}$ and $\widetilde{\beta}$ are linearly independent over $\Q$, the sets $\Delta(k)$ are pairwise disjoint. Hence
$\mathcal{M}$ is a Lie subalgebra of  $\mathcal{L}$. 
Since $\mathcal{M}_1$ is infinite-dimensional, by
Lemma \ref{wgraded}
$U(\mathcal{M})$ is not Noetherian and 
by  Lemma \ref{basic}  $U(\mathcal{L})$ is not Noetherian.
\end{proof}

\subsection{\it  Constructions of ideals in Lie algebras}

\

\medbreak
\noindent The next  two lemmas show that certain subspaces of a Lie algebra
are indeed ideals. Results of this kind are useful in the study of simple Lie algebras.

\smallbreak Let $\mathcal{L}$ be a Lie algebra.

\begin{lemma}\label{b+bb} \cite[Lemma 6]{Mathieu-jalg}
Let $\mathcal{A}$ and 
$\mathcal{B}$ be linear subspaces of $\Lc$
 such that $\mathcal{L}=\mathcal{A}+\mathcal{B}$
 and $[\mathcal{A},\mathcal{B}]\subset \mathcal{B}$.
Then $\mathcal{B}+ [\mathcal{B},\mathcal{B}]$ is an ideal of 
$\mathcal{L}$. \qed
\end{lemma}

Let $L$ be a linear subspace of $\Lc$. We say that 
an element $x\in\mathcal{L}$ is {\it locally $L$-nilpotent} if we have $\Ad(L)^{1+n}(x)=0$ for 
$n\gg  0$. Let $\Nc$ be the space of all locally $L$-nilpotent
elements and set $\mathcal{I}\coloneqq \cap_{N\geq 0}\,\Ad(L)^N(\mathcal{L})$.

\begin{lemma}\label{idealI} 

The subspace $\Nc$ is a Lie subalgebra and 
 $\mathcal I$ is a $\Nc$-submodule.

Consequently,  if $\Nc=\Lc$, the subspace $\mathcal I$ is an ideal.
 
\end{lemma}

\begin{proof} Let $x, y\in\mathcal{L}$. 
For any $N\geq 0$, we have

\begin{enumerate}[leftmargin=*,label=\rm{(\roman*)}]

\item
$\Ad(L)^{N}([x,y])
\subset \sum_{0 \leq k\leq N}
([\Ad(L)^{k}(x),\Ad(L)^{N-k}(y)])$,

\medbreak
\item
$[x,\Ad(L)^{N}(y)]
\subset \sum_{0 \leq k\leq N}
\Ad(L)^{N-k}([\Ad(L)^{k}(x), y]).$
\end{enumerate}
 
The first identity shows that 
$[\Nc,\Nc]\subset\Nc$, i.e. $\Nc$ is a Lie subalgebra.
The second identity shows that
$[\Nc,\Ic]\subset\Ic$, i.e. $\Ic$ is a $\Nc$-submodule.
\end{proof}

\subsection{\it  A dichotomy for the \texorpdfstring{$\Z$}{}-graded Lie algebras of rank \texorpdfstring{$\geq 2$}{}}

\

\medbreak
\noindent Let $\mathcal{L}$ be a simple 
$\Z$-graded Lie algebra of rank $\geq 2$.
We now define two hypothetical properties, and show that
any such $\mathcal{L}$ satisfies one of them.
By the end of the section it will be clear that these properties are mutually exclusive.

\medbreak
To start with, we define the notion of a string. Let $\widetilde\alpha\in Q$ and 
$\widetilde\beta\in \Delta$.
There are $a,b\in\Z\cup\{\pm\infty\}$
with $a<0<b$ such that

\begin{enumerate}[leftmargin=*,label=\rm{(\roman*)}]
\item $\widetilde\beta+k\widetilde\alpha$ belongs to $\Delta$ for any $k\in]a,b[$, but
	
\medbreak
\item neither 
$\widetilde\beta+a\widetilde\alpha$ nor
$\widetilde\beta+b\widetilde\alpha$ belongs to 
$\Delta$.
\end{enumerate}

\medbreak
\noindent The set 
$\{\widetilde\beta+k\widetilde\alpha\mid k\in]a,b[\}$ is called the 
\emph{$\widetilde\alpha$-string through $\widetilde\beta$}.

\medbreak
The first hypothetical property  
$(\mathcal{H}_{\operatorname{re}})$ is the following: 
\begin{align}\tag{$\mathcal{H}_{\operatorname{re}}$} 
\begin{aligned}
&\text{There exist }\widetilde{\alpha}\in\Delta_{\operatorname{re}}, \quad
\widetilde{\beta}\in \Delta, \quad 
\widetilde{\beta} \notin \Q.\widetilde{\alpha},  \text{ such that}
\\ &\text{the $\widetilde\alpha$-string through $\widetilde\beta$ is infinite.}
\end{aligned}
\end{align}
 The hypothesis $(\mathcal{H}_{\operatorname{re}})$ is  obviously stronger  than  $(\mathcal{H})$.

\medbreak
The second hypothetical property is the notion of weak integrability.
Following \cite{Mathieu-inv2}, we say that $\mathcal{L}$ is {\it weakly integrable} if,
for any $\widetilde{\alpha}\in\Delta_{\operatorname{re}}$, we have 
\begin{align*}
\bigcap_{n\geq 0}\,\Ad(\mathcal{L}^{\widetilde\alpha})^n(\mathcal{L})=0.
\end{align*}

\begin{lemma}\label{dicho} Let $\mathcal{L}$ be a simple $\Z$-graded algebra of
rank $\geq 2$. Then either

\begin{enumerate}[leftmargin=*,label=\rm{(\alph*)}]
\item   $\mathcal{L}$ satisfies the hypothesis 
$(\mathcal{H}_{\operatorname{re}})$, or

\item  $\mathcal{L}$ is weakly integrable.
\end{enumerate}
\end{lemma}

\begin{proof} Assuming that $\mathcal{L}$ does not satisfy
$(\mathcal{H}_{\operatorname{re}})$, we will prove that $\mathcal{L}$ is weakly
integrable.
Let $\widetilde{\alpha}\in\Delta_{\operatorname{re}}$,
let $\Nc$ be the space of locally 
$\Lc^{\widetilde\alpha}$-nilpotent elements and set

\begin{align*}\mathcal{I}=\cap_{N\geq 0}\,\Ad(\Lc^{\widetilde\alpha})^N(\mathcal{L})&&
\mathcal{A} =\oplus_{\widetilde{\beta}\in
\Q.\widetilde{\alpha}}
\,\mathcal{L}^{\widetilde\beta} &&\text{and}&&
\mathcal{B}=\bigoplus_{\widetilde{\beta}\notin\Q. 
\widetilde{\alpha}}
&\mathcal{L}^{\widetilde\beta}.
\end{align*}
 
 First, we prove that $\mathcal{L}=\Nc$. 
 For any $\widetilde\beta\notin \Q.\widetilde\alpha$,
 there is an integer $N>0$ such that 
 $\Lc^{\widetilde\beta+N\widetilde\alpha}=0$.
 It follows that 
 $\Ad^N(\Lc^{\widetilde\alpha})
 (\Lc^{\widetilde\beta})=0$. Therefore
 $\Nc$ contains $\Lc^{\widetilde\beta}$ for any
 $\widetilde\beta\notin \Q.\widetilde\alpha$, i.e.
 $\Nc$ contains $\Bc$. 
 By Lemma \ref{idealI}, $\Nc$ contains 
 $\Bc+[\Bc,\Bc]$. Since $[\Ac,\Bc]\subset\Bc$ and
 $\Lc=\Ac+\Bc$, the space $\Bc+[\Bc,\Bc]$ is an ideal
 by Lemma \ref{b+bb}. By simplicity of $\Lc$, we deduce that
$\Bc+[\Bc,\Bc]=\Lc$, and therefore we have  $\Nc=\mathcal{L}$. 
 
Next, we prove that $\mathcal{I}=0$. Let
$\widetilde\beta$ be a root which is not proportional to
$\widetilde\alpha$. There is an integer $N>0$ such that 
 $\Lc^{\widetilde\beta-N\widetilde\alpha}=0$.
 Hence $\Lc^{\widetilde\beta}$ is not contained in
 $\Ad^N(\Lc^{\widetilde\alpha})
 (\Lc)$.
 It follows that $\Ic$ does not contain
 $\Lc^{\widetilde\beta}$. However, by Lemma \ref{idealI},
  $\Ic$ is an ideal
 of $\Lc$. Hence we have 
$\mathcal{I}=0$.
In other words, 
$\mathcal{L}$ is weakly integrable.
\end{proof}

\subsection{\it  Non-Noetherianity for \texorpdfstring{$\Z$}{}-graded Lie algebras of rank \texorpdfstring{$\geq 2$}{}}

\

\medbreak
\begin{cor}\label{corrk2} Let $\mathcal{L}$ be a simple $\Z$-graded algebra of
rank $\geq 2$. 
 If $U(\mathcal{L})$ is  Noetherian, then $\mathcal{L}$ is finite-dimensional.
\end{cor}

\begin{proof} By Lemma \ref{dicho}, 
$\mathcal{L}$ satisfies the hypothesis 
$(\mathcal{H}_{\operatorname{re}})$ or $\mathcal{L}$ is weakly integrable.
In the first case,  $U(\mathcal{L})$ is not Noetherian by Lemma
\ref{H1}.

\medbreak
In the latter case, $\mathcal{L}$ is isomorphic to an affine Lie algebra or it has finite dimension by \cite[Theorem 4]{Mathieu-inv2}. But if $\Lc$ is an affine Lie algebra,
then it has an infinite-dimensional abelian  subalgebra, hence
$U(\mathcal{L})$ is not Noetherian.
\end{proof}

\section{Proof of the Main Result}

\subsection{\it The endomorphisms of simple \texorpdfstring{$\Z^n$}{}-graded modules}

\

\medbreak
\noindent Let $\mathcal{L}$ be a $\Z^n$-graded Lie algebra and
let $M$ be a simple $\Z^n$-graded module. 

\begin{lemma}\label{nonsimplemod} If $M$ is not simple
(as a non-graded module),
then there exists   $\theta\in\End_\mathcal{L}(M)$
which is  invertible and homogeneous of  degree ${\bf p}$ for some ${\bf p}\in \Z^n\backslash 0$.
\end{lemma}

\begin{proof} 
 Any $v\in M$ decomposes as $v=\sum_{\bf m}\,v_{\bf m}$ where 
$v_{\bf m}\in M_{\bf m}$. By definition the {\it support} of $v$ is the set
\begin{align*}
\Supp(v)\coloneqq \{{\bf n}\in\Z^n\mid v_{\bf n}\neq 0\}.
\end{align*}
Assume  that $M$ is not simple.
Let $v\in M\setminus 0$ be the generator of a proper submodule with a  support
of minimal cardinality. 
Since $M$ is  simple as a $\Z^n$-graded module, $v$ is not homogenous. Hence
$\Supp(v)$ contains distinct elements 
${\bf m}, {\bf n}$. 

\medbreak
We claim that $\Ann(v_{\bf n})\subset \Ann(v_{\bf m})$,
where $\Ann(m)$ denotes the annihilator of $m$ in 
$U(\Lc)$, for any $m\in M$. Since $\Ann(v_{\bf n})$ is
$\Z^n$-graded, it is enough to show that any homogenous
element $u\in \Ann(v_{\bf n})$ belongs to $\Ann(v_{\bf m})$.
Since $u.v_{\bf n}=0$, the support
$u . v$ lies in $({\bf d} + \Supp(v))\setminus\{{\bf d} + {\bf n}\}$, where ${\bf d}$ is the degree of $u$.
By minimality of the cardinality of 
$\Supp(v)$, we deduce that $u.v=0$ which proves the claim.

Hence there exists $\theta \in \End_{\mathcal{L}}(M)$  
mapping $v_{\bf n}$ to $v_{\bf m}$.
Clearly, $\theta$  is homogeneous of degree ${\bf p}={\bf m}-{\bf n} \in \Z^n$.
Since $\Ker \theta$ and $\Image\,\theta$ are graded submodules,  $\Ker \theta=0$ and
$\Image\,\theta=M$, hence $\theta$ is invertible.
\end{proof}

\subsection{\it  Simple \texorpdfstring{$\Z^n$}{}-graded Lie algebras which are not simple}

\

\medbreak
\noindent Let $\mathcal{L}$ be a $\Z^n$-graded Lie algebra.
The algebra of endomorphisms of the adjoint module is called the {\it centroid} of $\mathcal{L}$.

\begin{lemma}\label{centroid} If  the simple  $\Z^n$-graded Lie algebra 
$\mathcal{L}$ is not simple as a Lie algebra,
then it  contains an infinite-dimensional abelian subalgebra.
\end{lemma}

\begin{proof} By Lemma \ref{nonsimplemod}, there is 
an element $\theta \neq 0$ in the centroid which is homogeneous of 
degree ${\bf m }\in\Z^n\setminus 0$. 
Let $0 \neq x \in \Lc$ be an homogeneous element. Let
$\fm$ be the linear span of $\{\theta^{p}(x): p \in \Z\}$.
For  $p,q\in\Z$, we have 
\begin{align*}[\theta^{p}(x),\theta^{q}(x)]=\theta^{p+q} ([x,x])=0,
\end{align*}
\noindent hence  $\fm$  is a abelian subalgebra. Moreover the elements $\theta^{p}(x)$ are nonzero elements of different degrees, hence $\fm$ is infinite-dimensional.
\end{proof}

\begin{cor}\label{nonsimplecor} Assume that   the simple  $\Z^n$-graded Lie algebra 
$\mathcal{L}$ is not simple as a Lie algebra.
Then $U(\mathcal{L})$ is not Noetherian.
\end{cor}

\begin{proof} 
This is a consequence of  Lemmas \ref{centroid} and  \ref{basic} \ref{abelian-section}.
\end{proof}

\subsection{\it  Proof of the main result}

\

\medbreak
\noindent We can now prove the main Theorem of  this article.

\begin{mtheorem}
Let $\Lc$ be a simple $\Z^n$-graded  Lie algebra
of infinite dimension.  
Its enveloping algebra $U(\Lc)$ is not Noetherian.
\end{mtheorem}

\begin{proof} We can  assume that  $\mathcal{L}$ is simple as a Lie algebra, otherwise $U(\mathcal{L})$ is not Noetherian by Corollary
\ref{nonsimplecor}.

There exists ${\bf m}=(m_1,\dots,m_n)\neq 0$ such that
$\mathcal{L}_{\bf m}\neq 0$. Without loss of generality, we can assume that $m_1\neq 0$. Define the weakly $\Z$-graded Lie algebra
$\mathcal{L}'$ (which is $\mathcal{L}$  as Lie algebra) by the requirement that
\begin{align*}
\mathcal{L}'_m &= \bigoplus_{ (m_2,\dots,m_n)\in\Z^{n-1}} 
\mathcal{L}_{(m,m_2,\dots,m_n)}.
\end{align*}
  We can assume that all homogeneous components of
 $\mathcal{L}'$ are finite-dimen\-sional, otherwise 
 $U(\mathcal{L})$ is not Noetherian by Lemma
\ref{wgraded}.
 
 \medbreak
 Therefore $\mathcal{L}'$ is a simple $\Z$-graded Lie algebra. If $\mathcal{L}'$ has rank one,
$U(\mathcal{L})$ is not Noetherian by Corollaries
\ref{corV} and \ref{corS}. Otherwise 
$\mathcal{L}'$ has rank $\geq 2$ and $U(\mathcal{L})$ is not Noetherian by Corollary \ref{corrk2}.
 \end{proof} 

\begin{remark}
Let $\Lc$ be a  $\Z^n$-graded  Lie algebra. 
If $\Lc$ has a simple infinite-dimensional graded section, then
Theorem \ref{thm:graded-simple} implies that $U(\Lc)$ is not Noetherian.
In other words, if $U(\Lc)$ is Noetherian, then any simple graded section has finite dimension,
in particular any maximal graded ideal has finite codimension.
\end{remark}

\begin{remark} In the theorem,  we had assumed that all homogenous components of $\Lc$ are finite-dimensional.
In fact we can replace this condition by the weaker hypothesis that $\Lc\neq \Lc_0$, see Lemma \ref{wgraded}.
\end{remark}

\subsection*{Acknowledgements} 
The authors  thank Efim Zelmanov and Slava Futorny 
for the warm hospitality during their visit  to the 
Shenzhen International Center for Mathematics.


\begin{thebibliography}{AD}

\bibitem{Amayo-Stewart} R.~Amayo, I.~Stewart, Infinite-dimensional Lie algebras. Springer,  (1974).



\bibitem{Bourbaki}
N.~Bourbaki, Elements of mathematics. Lie groups and Lie algebras. Chapters 7--9. Transl. from the French by Andrew Pressley. Berlin: Springer (2005).

\bibitem{Brown-survey}
K.~A. Brown, \emph{Noetherian {H}opf algebras.} Turkish J. Math. \textbf{31}
(2007), 7--23.

\bibitem{Buzaglo} L. Buzaglo, \emph{Enveloping algebras of Krichever-Novikov algebras are not Noetherian}, 
Algebr. Represent. Theory \textbf{26},  2085--2111 (2023). 


\bibitem{Kac} V.G. Kac, 
\emph{Simple irreducible graded Lie algebras of finite growth}, 
Math. USSR-Izvestija \textbf{2} (1968), 1271-1311 (1970).


\bibitem{Kac-libro} V. G. Kac,  Infinite-dimensional Lie algebras. Third edition. Cambridge University Press, Cambridge, (1990).


\bibitem{Mathieu-jalg} O.~Mathieu, \emph{
On a problem of V. G. Kac: the classification of certain simple graded Lie algebras. (Sur un problème de V. G. Kac: La classification de certaines algèbres de Lie graduées simples).}
J. Algebra \textbf{102}, 505--536 (1986).

\bibitem{Mathieu-inv1} O.~Mathieu,
\emph{Classification des algèbres de Lie graduées simples de croissance $\leq1$. (Classification of simple graded Lie algebras of growth  $\leq1$)}.
Invent. Math. \textbf{86}, 371--426 (1986).

\bibitem{Mathieu-inv2} O.~Mathieu,
\emph{Classification of simple graded Lie algebras of finite growth},
Invent. Math. \textbf{108}, 455--519 (1992).


\bibitem{Sierra-Walton}
S.~J. Sierra and C.~Walton, \emph{The universal enveloping algebra of the {W}itt
algebra is not Noetherian}. Adv. Math. \textbf{262} (2014), 239--260.





\end{thebibliography}
\end{document}